\documentclass[11pt, reqno]{amsart}

\usepackage{graphicx}
\usepackage{amsmath,amssymb,mathtools,amsopn,amsfonts,amsthm}
\usepackage[shortlabels]{enumitem}
\usepackage{float}
\usepackage{bbm}
\usepackage{subcaption}
\usepackage[utf8]{inputenc}
\usepackage{eepic}
\usepackage[usenames, dvipsnames]{xcolor}
\usepackage{tikz}
\usetikzlibrary{positioning}
\usepackage{a4wide}
\usepackage{caption}

\usepackage{epsfig}
\usepackage{latexsym}
\usepackage[UKenglish]{babel}
\usepackage{verbatim}
\usepackage[initials]{amsrefs}
\usepackage[bookmarks]{hyperref}
\usetikzlibrary{decorations.markings}

\newtheorem{thm}{Theorem}[section]
\newtheorem{lma}[thm]{Lemma}

\newtheorem{cor}[thm]{Corollary}

\theoremstyle{definition}

\newtheorem{rem}[thm]{Remark}
\newtheorem{ques}[thm]{Question}

\newcommand{\R}{\mathbb{R}}

\newcommand{\N}{\mathbb{N}}

\providecommand{\norm}[1]{\lVert#1\rVert}

\newcommand{\I}{\mathcal{I}}

\renewcommand{\i}{\mathtt{i}}
\renewcommand{\j}{\mathtt{j}}

\newcommand{\g}{\mathbf{g}}
\renewcommand{\u}{\mathbf{u}}
\renewcommand{\v}{\mathbf{v}}
\newcommand{\w}{\mathbf{w}}
\newcommand{\ii}{\mathbf{i}}
\newcommand{\jj}{\mathbf{j}}

\newcommand{\G}{\mathcal{G}}

\newcommand{\hd}{\dim_\textup{H}}
\newcommand{\bd}{\dim_\textup{B}}
\newcommand{\ubd}{\overline{\dim}_\textup{B}}
\newcommand{\lbd}{\underline{\dim}_\textup{B}}

\newcommand{\Tau}{\mathcal{T}}

\title{Non-existence of the box dimension for dynamically invariant sets}

\author{Natalia Jurga} \address{Mathematical Institute, University of St Andrews, Scotland, KY16 9SS}
\email{naj1@st-andrews.ac.uk}

\thanks{The  author was  supported by an \emph{EPSRC Standard Grant} (EP/R015104/1). The author would like to express her gratitude to Ian Morris and Jonathan Fraser, whose interesting comments and suggestions improved the paper, as well as the anonymous referee who corrected an error in the proof of Lemma 3.1. The author also thanks Pablo Shmerkin, whose question stimulated this work.}

\begin{document}

\maketitle

\begin{abstract}
One of the key challenges in the dimension theory of smooth dynamical systems is in establishing whether or not the Hausdorff, lower and upper box dimensions coincide for invariant sets. For sets invariant under conformal dynamics, these three dimensions always coincide. On the other hand, considerable attention has been given to examples of sets invariant under non-conformal dynamics whose Hausdorff and box dimensions do not coincide. These constructions exploit the fact that the Hausdorff and box dimensions quantify size in fundamentally different ways, the former in terms of covers by sets of varying diameters and the latter in terms of covers by sets of fixed diameters. In this article we construct the first example of a dynamically invariant set with distinct lower and upper box dimensions. Heuristically, this describes that if size is quantified in terms of covers by sets of equal diameters, a dynamically invariant set can appear bigger when viewed at certain resolutions than at others.
\end{abstract}

\section{Introduction}

The dimension theory of dynamical systems is the study of the complexity of sets and measures which remain invariant under dynamics, from a dimension theoretic point of view. This branch of dynamical systems has its foundations in the seminal work of Bowen \cite{bowen} on the dimension of quasicircles and Ruelle \cite{ruelle} on the dimension of conformal repellers, and has since developed into an independent field of research which continues to receive noteworthy attention in the literature \cite{bhr,cpz,ds}. For an overview of this extensive field, see the monographs of Pesin \cite{pesin} and Barreira \cite{barreira-book} and the surveys \cite{bg, cp,sw}.

The most common ways of measuring the dimension of invariant sets are through the Hausdorff and the lower and upper box dimensions, which quantify the complexity of the set in related but subtly distinct ways. Roughly speaking, the Hausdorff dimension  measures how efficiently the set can be covered by sets of arbitrarily small size, whereas the lower and upper box dimensions measure this in terms of covers by sets of uniform size, along the scales for which this can be done in the most and least efficient way, respectively. Given a subset $E$ of a separable metric space $X$, the lower and upper box dimensions are defined by
\[
 \lbd E  = \liminf_{\delta \to 0} \frac{\log N_\delta(E)}{-\log \delta} \qquad \text {and} \qquad \ubd E = \limsup_{\delta \to 0} \frac{\log N_\delta(E)}{-\log \delta},
\]
respectively, where $N_\delta(E)$ denotes the smallest number of sets of diameter $\delta>0$ required to cover $E$.  If the lower and upper box dimensions coincide we call the common value the box dimension, written $\bd$, otherwise we say that the box dimension does not exist.  

For any subset $E \subseteq X$, 
\begin{equation} \label{hb}
\hd E \leq \lbd E \leq  \ubd E
\end{equation}
where $\hd$ denotes the Hausdorff dimension. A priori each inequality may or may not be strict. However, when  $E$ is invariant under a smooth mapping $f$, the additional structure imposed by the dynamical invariance of $E$ means that certain properties of $f$ can either force some degree of homogeneity or, on the contrary,  inhomogeneity across the set, forcing  equalities or strict inequalities in \eqref{hb} respectively.  Characterising which properties of $f$  
imply or preclude equalities in \eqref{hb} is one of the key challenges in dimension theory.


A common feature in the dimension theory of smooth \emph{conformal} dynamics is the coincidence of the Hausdorff and lower and upper box dimensions for invariant sets. For example, in the setting of smooth expanding maps, the following result pertains to a more general result which was obtained independently by Gatzouras and Peres \cite{gp} and Barreira \cite{barreira-paper},  generalising previous results of Falconer \cite{falc-paper}.

\begin{thm}[\cite{gp,barreira-paper}]\label{conformal}
Suppose $f:M \to M$ is a $C^1$ map of a Riemannian manifold $M$ and that $\Lambda=f(\Lambda)$ is a compact set such that $f^{-1}(\Lambda)\cap U \subset \Lambda$ for some open neighbourhood $U$ of $\Lambda$. Additionally, assume that
\begin{itemize}
\item $f$ is \textbf{conformal}: for each $x \in M$ the derivative $d_xf$ is a scalar multiple of an isometry,
\item $f$ is \textbf{expanding} on $\Lambda$: there exist constants $C>0$, $\lambda>1$ such that for all $x \in \Lambda$ and $u$ in the tangent space $T_xM$, 
$$\norm{d_xf^nu} \geq C\lambda^n\norm{u}.$$ 
\end{itemize}
Then for any compact set $F=f(F) \subset \Lambda$,
$$\lbd F=\ubd F=\hd F.$$
\end{thm}


Similar results hold in the setting of smooth diffeomorphisms. For example, if $f:M \to M$ is a topologically transitive $C^1$ diffeomorphism with a basic set $\Lambda$, and $f$ is conformal on $\Lambda$ then $\hd \Lambda=\lbd \Lambda=\ubd \Lambda$ \cite{pesin, barreira-paper} and an analogous statement holds for the dimensions of the intersections of $\Lambda$ with its local stable and unstable manifolds \cite{takens, pv}. 


In contrast, in the realm of smooth \emph{non-conformal} dynamical systems, coincidence of the Hausdorff and box dimensions is no longer a universal trait of invariant sets. Indeed, examples of invariant sets with distinct Hausdorff and box dimensions have attracted enormous attention \cite{neun, bedford, mcmullen,kp, lg, pw} and discussion in surveys \cite{cp, bg,fraser}. This type of dimension gap result exploits the fact that the Hausdorff dimension quantifies the size of the set in terms of covers by sets of varying diameters rather than fixed diameters which are used by the box dimension. Indeed invariant sets of certain non-conformal dynamics will contain long, thin and well-aligned copies of itself, meaning that covering by sets of varying diameter is often more efficient, inducing this type of dimension gap. However, 
surprisingly there seems to be no mention in the literature of the possibility of a dynamically invariant set with \emph{distinct lower and upper box dimensions}.
Our main result demonstrates the existence of such sets.

\vbox{
\begin{thm} \label{main}
	There exist integers $n>m\geq 2$ and a compact subset of the torus $F \subset \mathbb{T}^2$ such that $F$ is invariant, $F=T(F)$, under the expanding toral endomorphism
$$T(x,y)=(mx \, \mathrm{mod}\, 1, ny\, \mathrm{mod} \,1)$$
and 
$$\lbd F<\ubd F.$$
In particular, the box dimension of $F$ does not exist. 
\end{thm}}

Since $n>m$, $T$ is a non-conformal map. Well-known examples from the literature, such as Bedford-McMullen carpets \cite{fraser}, demonstrate that equality of the Hausdorff and box dimensions is not guaranteed in Theorem \ref{conformal} if the assumption of conformality is dropped. Furthermore, Theorem \ref{main} indicates that the lower and upper box dimensions need not coincide either in Theorem \ref{conformal} if the assumption of conformality is dropped. This is arguably a more striking type of dimension gap since, while it is easy to see that sets invariant under non-conformal dynamics may cease to be homogeneous in space, which is captured by the possibility of distinct Hausdorff and box dimensions, one would expect the dynamical invariance to at least force homogeneity in scale, but our result demonstrates that this too can fail. In particular Theorem \ref{main} describes that, when measuring size in terms of covers by sets of equal diameter, a dynamically invariant set can sometimes appear bigger and at other times appear smaller depending on the ``resolution'' we are viewing it at. We highlight that our construction is also significantly more involved than standard examples of invariant sets with distinct Hausdorff and box dimensions, such as Bedford-McMullen carpets.

The dynamics of $T$ on the invariant set $F$, which will be constructed in \S \ref{constr}, has two key features which in conjunction induce distinct box dimensions. Firstly, the non-conformality of $T$ causes the box dimensions of $F$ to be sensitive to the length of time it takes for an orbit of $T$ to move from a subset $A \subset F$ which is ``entropy maximising''  for the dynamics of $T$ to a subset $B$ which is ``entropy maximising'' for the dynamics of the projection $x \mapsto mx\, \textnormal{mod}\, 1$ of $T$. Secondly the dynamics on $F$, which can be modelled by a topologically mixing  \emph{coded subshift} \cite{bh} on an appropriate symbolic space, has the property that the length of time it takes an orbit of $T$ to move from $A$ to $B$ is highly dependent on how long the orbit has spent in $A$. In particular, the dynamics fails to satisfy most forms of specification \cite{kwietniak}. The resolution at which $F$ is viewed determines how long the orbits of points of interest (for the dimension estimates at that particular resolution) spend in $A$, and combined with the properties mentioned above this forces distinct box dimensions. 


Finally, we discuss some connections between Theorem \ref{main} and the literature on self-affine and sub-self-affine sets. Let $\{S_i:\R^d \to \R^d\}_{i=1}^N$ be a collection of affine contractions, i.e. $S_i(\cdot)=A_i(\cdot)+\mathbf{t}_i$ for each $1\leq i \leq N$ where $A_i \in GL(d,\R)$ with Euclidean norm $\norm{A_i}<1$ and $\mathbf{t}_i \in \R^d$. We call $\{S_i\}_{i=1}^N$ an affine iterated function system. A \emph{sub-self-affine set} \cite{kv}, is a non-empty, compact set $E \subset \R^d$ such that
\begin{equation} \label{ssa}
E \subseteq \bigcup_{i=1}^N S_i(E).
\end{equation}
If \eqref{ssa} is an equality then $E$ is called a \emph{self-affine set}, in particular every self-affine set is an example of a sub-self-affine set. Every affine iterated function system admits a unique self-affine set. However, there are infinitely many sub-self-affine sets which are not self-affine. Indeed, the unique self-affine set is the image of the full shift $\{1,\ldots, N\}^\N$ under an appropriate projection induced from the family $\{S_i\}_{i=1}^N$, whereas sub-self-affine sets are in one-to-one correspondence with the projections of subshifts of the full shift. 
Under suitable ``separation conditions'' on $\{S_i\}_{i=1}^N$, any sub-self-affine set $E$ satisfies $f(E) \subseteq E$ for an appropriate piecewise expanding map $f$ given by the inverses of the contractions. The set $F$ which will be constructed in \S \ref{constr} to prove Theorem \ref{main} is a sub-self-affine set (which is not self-affine) for the affine iterated function system induced from the inverse branches of $T$. 

The dimension theory of self-affine sets has been an active topic of research since the 1980s and substantial progress has been made in recent years. Sub-self-affine sets were introduced by K\"aenm\"aki and Vilppolainen \cite{kv} as natural analogues of sub-self-similar sets which were studied earlier by Falconer \cite{sss}. 
It is known by the results of Falconer \cite{falconer} and K\"aenm\"aki and Vilppolainen \cite{kv} that the box dimension of a generic sub-self-affine sets exists, moreover this has been verified for large explicit families of planar self-affine sets \cite{bhr}.  However, the following question was open till now.

\begin{ques} \label{conj} Does the box dimension of every (sub-)self-affine set exist?
\end{ques}

The version of the above question for self-affine sets is a folklore open question within the fractal geometry community, to which the answer is widely conjectured to be affirmative. In contrast, a corollary of our main result is that the answer to Question \ref{conj} for general sub-self-affine sets is negative.

\begin{cor}
There exist sub-self-affine sets whose box dimension does not exist.
\end{cor}

\noindent \textbf{Organisation of paper.} In \S \ref{constr} we construct the set $F$ and its underlying subshift $\Sigma$ and offer some heuristic reasoning behind Theorem \ref{main}. \S \ref{entropy} contains entropy estimates. In \S \ref{dim} we introduce the scales for the lower and upper box dimension computations and prove Theorem \ref{main}. \S \ref{further} contains some questions for further investigation.

\section{Construction of $(\times m, \times n)$-invariant set.}\label{constr} 

Fix $m=2$, $n=12$. Let $\Delta=\{(a,b)\; : \; 1 \leq a \leq 2, \; 1 \leq b \leq 12, \; a,b \in \mathbb{N}\}.$  For any $(a,b) \in \Delta$ define the contraction $S_{(a,b)}:[0,1]^2 \to [0,1]^2$ as
$$S_{(a,b)}\left(x,y\right)=\left(\frac{x}{2}+ \frac{a-1}{2}, \frac{y}{12}+\frac{b-1}{12}\right) $$
which are the partial inverses of $T$. If $\i, \j \in \Delta^\N$ with $\i \neq \j$ we let $\i \wedge \j$ denote the longest common prefix to $\i$ and $\j$, and denote its length by $|\i \wedge \j|$. We equip $\Delta^\N$ with the metric 
\begin{equation*}
d(\i,\j)= \begin{cases} 
\frac{1}{2^{|\i \wedge \j|}} &  \textnormal{if} \; \i \neq \j \\
0 & \textnormal{if} \; \i=\j
\end{cases} \end{equation*}
 The set $F$ that satisfies Theorem \ref{main} will be given by the projection of a set $\Sigma \subseteq \Delta^\N$ under the continuous and surjective (but not injective) coding map $\Pi:\Delta^\N \to [0,1]^2$ given by
$$\Pi\left((a_1,b_1)(a_2,b_2) \ldots\right):=\lim_{n \to \infty} S_{(a_1,b_1) \cdots (a_n,b_n)}(0)$$
where $S_{(a_1,b_1) \cdots (a_n,b_n)}$ denotes the composition $S_{(a_1,b_1)}\circ \cdots \circ S_{(a_n,b_n)}$.

Let $\Omega=\{(1,i)\}_{i=3}^{12}$. For each $N \in \N$ let $\Omega^N$ denote words of length $N$ with symbols in $\Omega$, and $\Omega^\N$ the set of infinite sequences with symbols in $\Omega$. Given any $(a,b) \in \Delta$, $(a,b)^n$ denotes the word $(a,b)(a,b) \cdots (a,b)$ of length $n$. Define $\mathcal{C}$ to be the collection of words 
\begin{equation} \mathcal{C}:=\{(1,1), (2,1)\} \cup \bigcup_{N=1}^\infty \bigcup_{\w \in \Omega^{13^N}} \{\w(1,2)^{13^N}\}$$ and  $$B:=\left\{\u\u_1\u_2 \u_3\ldots \; \;: \;\;  \textnormal{$\u_i \in \mathcal{C}$ for all $i \in \N$,} \; \textnormal{$\u$ is a suffix of some word in $\mathcal{C}$}\right\}. \label{codes} \end{equation}
Then we define the sequence space $\Sigma=\overline{B}$.\footnote{The set of accumulation points $\Sigma \setminus B$ will turn out to be unimportant for our analysis, but for the readers convenience we provide a description of this set in \eqref{boundary}.} Equivalently $B$ can be understood as the set of all infinite sequences which label a one-sided infinite path on the directed graph $G$ in Figure \ref{G}. $G$ is called the presentation of $\Sigma$.
\begin{figure}[H]
\centering
\begin{subfigure}{0.5\textwidth}
\centering
   \begin{tikzpicture}[
roundnode/.style={circle, draw=black!60, fill=black!5, thick, minimum size=2mm},]
        \node [roundnode] (1) {$v$} ;

 \path[->] (1) edge  [loop right] [dashed] node[auto, yshift=-3pt]  {\small $\substack{ \w\,(1,2)^z \\ (\w \in \Omega^z, \; z=13^N)}$} ();
        \path[->] (1) edge  [in=110,out=150,loop] node[auto, xshift=10pt, yshift=-3pt] {\tiny $(1,1)$} ();
        \path[->] (1) edge  [in=-110,out=-150,loop] node[auto, xshift=-12pt, yshift=-14pt] {\tiny $(2,1)$} ();
    \end{tikzpicture}

\end{subfigure}%
\begin{subfigure}{0.5\textwidth}
\centering
\begin{tikzpicture}[xscale=2, yscale=0.33333]
\draw[xstep=1cm,ystep=1cm,color=black] (0,0) grid (2,12);
\filldraw[fill=yellow!50!white!50!, draw=black] (0,0) rectangle (1,1);
\filldraw[fill=yellow!50!white!50!, draw=black] (0,1) rectangle (1,2);
\filldraw[fill=yellow, draw=black] (0,2) rectangle (1,3);
\filldraw[fill=yellow, draw=black] (0,3) rectangle (1,4);
\filldraw[fill=yellow, draw=black] (0,4) rectangle (1,5);
\filldraw[fill=yellow, draw=black] (0,5) rectangle (1,6);
\filldraw[fill=yellow, draw=black] (0,6) rectangle (1,7);
\filldraw[fill=yellow, draw=black] (0,7) rectangle (1,8);
\filldraw[fill=yellow, draw=black] (0,8) rectangle (1,9);
\filldraw[fill=yellow, draw=black] (0,9) rectangle (1,10);
\filldraw[fill=yellow, draw=black] (0,10) rectangle (1,11);
\filldraw[fill=yellow, draw=black] (0,11) rectangle (1,12);
\filldraw[fill=yellow!50!white!50!, draw=black] (1,0) rectangle (2,1);
\node at (0.5,+0.5) {\tiny$(1,1)$};
\node at (0.5,+1.5) {\tiny$(1,2)$};
\node at (0.5,+2.5) {\tiny$(1,3)$};
\node at (0.5,+3.5) {\tiny$(1,4)$};
\node at (0.5,+4.5) {\tiny$(1,5)$};
\node at (0.5,+5.5) {\tiny$(1,6)$};
\node at (0.5,+6.5) {\tiny$(1,7)$};
\node at (0.5,+7.5) {\tiny$(1,8)$};
\node at (0.5,+8.5) {\tiny$(1,9)$};
\node at (0.5,+9.5) {\tiny$(1,10)$};
\node at (0.5,+10.5) {\tiny$(1,11)$};
\node at (0.5,+11.5) {\tiny$(1,12)$};
\node at (1.5,+0.5) {\tiny$(2,1)$};
\end{tikzpicture}
\end{subfigure} 
\caption{Left: the presentation $G$ of $\Sigma$. The dashed loop indicates that for each $N \in \N$ and $\w \in \Omega^{13^N}$ there is a path of length $2 \cdot 13^N$ which begins and ends at $v$ such that its sequence of labels reads $\w (1,2)^{13^N}$. Right: Images of $[0,1]^2$ under $S_{(a,b)}$, for each $(a,b)$ that labels some edge in $G$. The darker coloured rectangles correspond to $S_{(a,b)}([0,1]^2)$ for $(a,b) \in \Omega$.}\label{G}
\end{figure}
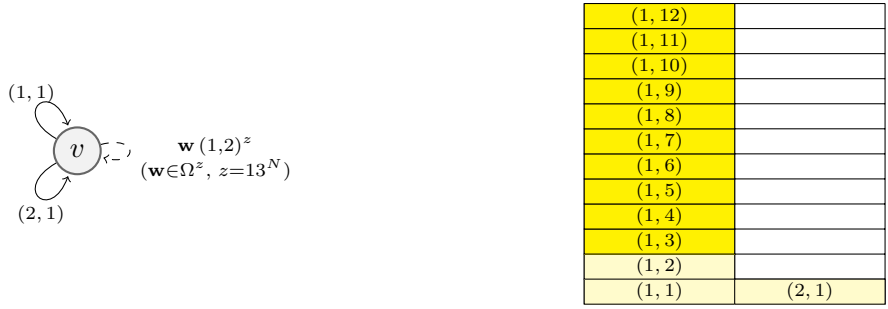

It is easy to check that $\sigma(\Sigma)=\Sigma$ where $\sigma:\Sigma \to \Sigma$ denotes the left shift map. In particular, $\Sigma$ is an example of a \emph{coded subshift}, meaning a subshift which can be expressed as the closure of the space of all infinite paths on a path-connected (possibly infinite) graph, which were first introduced by Blanchard and Hansel \cite{bh}. Note that whenever this graph is finite, its coded subshift is necessarily sofic, and that any $(\times m, \times n)$-invariant set which can be modelled by a sofic shift has a well-defined box dimension which can be explicitly computed \cite{kp,fj}.  Finally we set $F=\Pi(\Sigma)$, noting that $F=T(F)$ since $\sigma(\Sigma)=\Sigma$ and $\Pi \circ \sigma=T \circ \Pi$. From this it is easy to see that $F$ is a sub-self-affine set for the iterated function system $\{S_{(a,b)}: (a,b) \in \Delta\}$.

While of course it will  be necessary to cover the entirety of $F$ and obtain bounds on the size of this cover at different scales, the proof of Theorem \ref{main} will essentially boil down to the asymptotic difference that emerges between (a) the size of the cover, by squares of side $12^{-13^N}$, of the intersection of $F$ with the collection of rectangles $\{S_\ii([0,1]^2): \ii \in \Omega^{13^N}\} $  and (b) the size of the cover, by squares of side $12^{-13^{N-1/2}}$, of the intersection of $F$ with the collection of rectangles $\{S_\ii([0,1]^2): \ii \in \Omega^{13^{N-1/2}}\}$. 

Roughly speaking, $F$ occupies a large proportion of the width of each rectangle $S_\ii([0,1]^2)$ in case (a).   Such a rectangle has width $2^{-13^N}$ and height $12^{-13^N}$ (which equals the sidelength of squares in the cover). For any $\ii \in \Omega^{13^N}$ and $\jj \in \{(1,1),(2,1)\}^{13^N(\log 12/\log 2-2)}$, $\ii(1,2)^{13^N}\jj$ constitutes a legal word in $\Sigma$ and each $S_{\ii(1,2)^{13^N}\jj}([0,1]^2)$ has width roughly $12^{-13^N}$ (which equals the sidelength of squares in the cover), therefore $S_\ii([0,1]^2)$ requires roughly $2^{13^N(\log 12/\log 2-2)}$ squares to cover it. Importantly, this is a positive power of $12^{13^N}$, which indicates ``growth'' in dimension. 

On the other hand,  $F$ occupies a very thin proportion of the width of each rectangle $S_\ii([0,1]^2)$ in case (b). Each such rectangle has width  $2^{-13^{N-1/2}} $ and height $12^{-13^{N-1/2}}$ (which equals the sidelength of squares in this cover). Any $\i \in \Sigma$ which begins with a word in $\Omega^{13^{N-1/2}}$ can be written as $\i=\ii\jj\j$ for $\ii \in \Omega^{13^{N-1/2}}$, $\jj =(1,b_1) \cdots (1,b_{13^N})$ and some infinite word $\j \in \Sigma$. In particular, any point in $F \cap S_\ii([0,1]^2)$ belongs to $S_{\ii\jj}([0,1]^2)$ which has width \emph{less than} $12^{-13^{N-1/2}}$. In particular, only one square of sidelength $12^{-13^{N-1/2}}$ is required to cover $S_\ii([0,1]^2)$, meaning no further ``growth'' in dimension at this scale.

\subsection{Notation.} For any $N \in \N$ we let $\Sigma_N$ denote the subwords of sequences in $\Sigma$ of length $N$. Finite words in $\bigcup_{N=1}^\infty \Sigma_N$ will be denoted in bold, using notation such as $\ii$ or $\jj$ whereas infinite words in $\Sigma$ will be denoted using typewriter notation such as $\i$ and $\j$. For integers $n \geq 1$, and infinite sequences $\i=(a_1,b_1)(a_2,b_2) \cdots$, $\i|n$ denotes the truncation of $\i$ to its first $n$ symbols $\i|n=(a_1,b_1) \cdots (a_n,b_n)$. The same notation is used for truncations of finite words $\ii=(a_1,b_1) \cdots (a_m,b_m)$ to its first $n$ symbols $\ii|n=(a_1,b_1) \cdots (a_n,b_n)$  when $m \geq n$. For any finite word $\ii=(a_1,b_1) \cdots (a_n,b_n)$, its length is denoted by $|\ii|=n$. Given any $(a,b) \in \Delta$, $(a,b)^\infty$ denotes the infinite word $(a,b)(a,b) \ldots$. For any finite word $\ii$ we denote the cylinder set by $[\ii]:=\{\i \in \Sigma: \i|n=\ii\}$. We let $\emptyset$ denote the empty word.

To avoid profusion of constants, we write $A \lesssim B$ if $A \leq cB$ for some universal constant $c>0$. We write $A \lesssim_\varepsilon B$ if $A \leq c_\varepsilon B$ for all $\varepsilon>0$ where the constant $c_\varepsilon$ depends on $\varepsilon$. We write $A \gtrsim B$ if $B \lesssim A$ and $A \approx B$ if both $A \lesssim B$ and $B \lesssim A$, and define the notation $A \gtrsim_\varepsilon B$ and $A \approx_\varepsilon B$ analogously. 

\section{Entropy estimates}\label{entropy}

In this section we obtain estimates on the entropy of important subsets of $\Sigma$. Let $\G_N$ be the words in $\Sigma_N$ which label a path that starts and ends at the vertex $v$ of the graph $G$ in Figure \ref{G}. Define
$$h(\G):=\limsup_{N \to \infty} \frac{1}{N} \log \# \G_N$$
where $\# \G_N$ denotes the cardinality of $\G_N$.

\begin{lma}
 $h(\G) \leq \log 4$.
\end{lma}

\begin{proof}
Fix $N \in \N$. Given a word in $\G_N$, let $c$ denote the number of symbols belonging to $\Omega$ and $a$ denote the number of symbols belonging to $\{(1,1), (2,1)\}$, noting that
\begin{enumerate}[(a)]
\item $2c+a=N$ and 
\item $c=\sum_{i=1}^j 13^{n_i}$ for some integers $n_1, \ldots n_j$.
\end{enumerate}
Fix $0 \leq a \leq N$ and let $\mathcal{S}_c$ be the set of possible ways that $c=\frac{N-a}{2}$ can be written as an ordered sum $c=\sum_{i=1}^j 13^{n_i}$. By ordered sum, we mean that if $(n_1', \ldots, n_j')$ is a permutation of $(n_1, \ldots, n_j)$ such that $(n_1', \ldots, n_j') \neq (n_1, \ldots, n_j)$ then $\sum_{i=1}^j 13^{n_i}$ is considered a distinct way of writing $c$ as a sum of powers of 13. Observe that $j \leq \frac{c}{13}$ (eg. consider writing $c= 13 \cdot \frac{c}{13}$ when $c$ is a multiple of $13$).

We begin by bounding $\# \mathcal{S}_c \leq 2^{\frac{c}{13}-1}$. Recall that any $n \in \N$ can be expressed in $2^{n-1}$ ways as an ordered sum of one or more positive integers. Moreover, $\# \mathcal{S}_c$ is clearly bounded above by the number of ways that $c/13$ can be decomposed into an ordered sum $\sum_{i=1}^\ell p_i$ for some positive integers $p_1, \ldots, p_\ell$. Hence $\# \mathcal{S}_c \leq2^{\frac{c}{13}-1}$.
Now let us return to considering a word in $\mathcal{G}_N$. Following each substring of symbols from $\Omega$, there is a tail of the same length consisting of $(1,2)$'s. The $a$ symbols from $\{(1,1), (2,1)\}$ can either be placed directly after any of these tails or at the beginning of the word. Therefore assuming that the string  contains $c=\frac{N-a}{2}$ symbols from $\Omega$ in blocks of lengths $13^{n_1}, \ldots, 13^{n_j}$ (so $c=\sum_{i=1}^{j} 13^{n_i}$) it follows that there are ${a+j \choose j}$ ways in which the $a$ symbols from $\{(1,1), (2,1)\}$ can be distributed. Bounding this above by the central binomial term and using the bounds ${2K \choose K} \leq 4^K$ and $j \leq \frac{c}{13}$ we obtain ${a+j \choose j} \leq 2^{a+\frac{N-a}{2\cdot13}}$. Hence
\begin{eqnarray*}
\# \G_N &\leq& \sum_{a=0}^N \#\mathcal{S}_{\frac{N-a}{2}}2^{a+\frac{N-a}{2\cdot13}} 10^{\frac{N-a}{2}}2^{a} \\
&\leq& \sum_{a=0}^N 2^{2a+\frac{N-a}{2}(2/13+\log_210)}\\
&=& \frac{2^{2(N+1)}-2^{(2/13+\log_210)(N+1)/2}}{2^2-2^{(2/13+\log_210)/2} }\lesssim 4^N
\end{eqnarray*}
since $(2/13+\log_210)/2<2$, completing the proof of the lemma.

\end{proof}

Let $\I_N$ be the words in $\Sigma_N$ which label a path that ends at $v$ in the graph $G$ in Figure \ref{G}. Clearly $\G_N \subseteq \I_N$. Denoting $\I^*=\bigcup_{N=1}^\infty \I_N$ and $\Omega^*=  \bigcup_{N=1}^\infty \Omega^N$ observe that
\begin{equation}
\Sigma \setminus B=\{\u\mathtt{w}: \u \in \I^* \cup \emptyset, \mathtt{w} \in \Omega^\N\} \cup \{\w(1,2)^\infty: \w \in \Omega^* \cup \emptyset\}.
\label{boundary} \end{equation}
Define
$$h(\I)= \limsup_{N \to \infty} \frac{1}{N} \log \# \I_N.$$

\begin{lma} \label{I}
$h(\I)\leq \log 4.$
\end{lma}

\begin{proof}
Fix $N \in \N$. Note that any word in $\I_N \setminus \G_N$ is either of the form
\begin{enumerate}[(a)]
\item $(1,2)^z\g$ for $\g \in \G_{N-z}$ or
\item $\w(1,2)^z\g$ for $z=13^k$ for some $k \in \N$, $\w \in \Omega^w$ where $0<w< z$ and $\g \in \G_{N-z-w}$.
\end{enumerate}
Fix any $\varepsilon>0$. The number of words of the form (a) is
$$\sum_{z=1}^N \# \G_{N-z} \lesssim_\varepsilon e^{N(h(\G)+\varepsilon)} =(4e^\varepsilon)^N.$$

The number of words of the form (b) is
\begin{eqnarray*}
\sum_{z=13^k<N} \sum_{w=1}^{\min\{z-1,N-z\}}10^{w}\# \G_{N-z-w} &\lesssim_\varepsilon&
\sum_{z=13^k<N} \sum_{w=1}^{\min\{z-1,N-z\}}10^{w}(4e^{\varepsilon})^{N-z-w}\\
&  \lesssim& \sum_{z=13^k<N}  \left(\frac{10}{4}\right)^{\min\{z-1,N-z\}} (4e^\varepsilon)^{N-z}.\end{eqnarray*}
Since
$$\sum_{z=13^k<\frac{N}{2}}  \left(\frac{10}{4}\right)^{\min\{z-1,N-z\}} (4e^\varepsilon)^{N-z}= \sum_{z=13^k<\frac{N}{2}}10^{z-1}4^{N-2z+1}e^{\varepsilon(N-z)} \lesssim_\varepsilon (4e^{2\varepsilon})^N,$$
and
$$\sum_{\frac{N}{2} \leq z=13^k<N }  \left(\frac{10}{4}\right)^{\min\{z-1,N-z\}} (4e^\varepsilon)^{N-z}=\sum_{\frac{N}{2} \leq z=13^k<N } (10e^\varepsilon)^{N-z} \lesssim_\varepsilon (10e^{2\varepsilon})^{\frac{N}{2}}< 4^N$$
for sufficiently small $\varepsilon$, we have that 
$$\# \I_N \lesssim_\varepsilon (4e^{2\varepsilon})^N.$$
Since $\varepsilon>0$ was arbitrary, the proof is complete.

\end{proof}

\section{Dimension estimates}\label{dim}

In this section, we introduce the sequences of scales which will be used for the lower and upper box dimension estimates, and prove Theorem \ref{main}. We also show how the proof of Theorem \ref{main} can be used to construct an infinitely generated self-affine set whose box dimension does not exist.

Let $\delta>0$. We let $k(\delta)$ denote the unique positive integer satisfying $12^{-k(\delta)} \leq \delta < 12^{1-k(\delta)} $ and $l(\delta)$ denote the unique positive integer satisfying $2^{-l(\delta)} \leq \delta < 2^{1-l(\delta)}$, noting that $k(\delta) < l(\delta)$ for sufficiently small $\delta$. By definition $l(\delta) = \lceil \frac{-\log \delta}{\log 2}\rceil$ and $k(\delta) = \lceil\frac{-\log \delta}{\log 12}\rceil$.

 Define the projection $\pi: \Delta^\N \to \{1,2\}^\N$ by $\pi((a_1,b_1)(a_2,b_2) \ldots)=(a_1a_2\ldots)$. For $\ii \in \Sigma_k$ and $l>k$ define 
\begin{equation}\label{M}
M(\ii,l)=\# \pi(\jj \in \Sigma_{l}: \jj|k=\ii).
\end{equation}

Our general covering strategy at each scale $\delta$ can now be described as follows. For each $\ii \in \Sigma_{k(\delta)}$ observe that $S_\ii([0,1]^2)$ is a rectangle of height $\frac{1}{12^{k(\delta)}} \approx \delta$. In particular $N_\delta(\Pi(\Sigma)) \approx \sum_{\ii \in \Sigma_{k(\delta)}} N_\delta(\Pi([\ii]))$. Notice that for each $\jj \in \Sigma_{l(\delta)}$, $S_{\jj}([0,1]^2)$ has width $\frac{1}{2^{l(\delta)}} \approx \delta$. Therefore for each $\ii \in \Sigma_{k(\delta)}$ we cover each projected cylinder $\Pi([\ii])$ independently by considering how many level $l(\delta)$ columns contain part of the set $\Pi(\Sigma)$ inside $\Pi([\ii])$. Since by definition the number of such columns is given by $M(\ii, l(\delta))$ we obtain
$$N_\delta(\Pi(\Sigma)) \approx \sum_{\ii \in \Sigma_{k(\delta)}}N_\delta(\Pi([\ii])) \approx \sum_{\ii \in \Sigma_{k(\delta)}}M(\ii, l(\delta)).$$

Define the null sequence $\{\delta_N\}_{N \in \N}$ by $\delta_N=\frac{1}{12^{13^N}}$, noting that $k(\delta_N)=13^N$ and $l(\delta_N)=\lceil13^N \frac{\log 12}{\log 2}\rceil$. Also define the null sequence $\{\delta_N'\}_{N \in \N}$ by $\delta_N'=\frac{1}{12^{13^{N-\frac{1}{2}}}}$, noting that $k(\delta_N')=\lceil13^{N-\frac{1}{2}}\rceil$ and $l(\delta_N')=\lceil13^{N-\frac{1}{2}}\frac{\log 12}{\log 2}\rceil$. 

In this section we will prove that

\begin{equation}\label{ineq}
\limsup_{N \to \infty} \frac{\log N_{\delta_N}(\Pi(\Sigma))}{-\log \delta_N}>\liminf_{N \to \infty} \frac{\log N_{\delta_N'}(\Pi(\Sigma))}{-\log \delta_N'}.
\end{equation}

Theorem \ref{main} will follow from \eqref{ineq} since it implies that $\ubd \Pi(\Sigma)> \lbd \Pi(\Sigma)$.

\begin{lma}[Scales with large dimension] \label{high}
$$\limsup_{N \to \infty} \frac{\log N_{\delta_N}(\Pi(\Sigma))}{-\log \delta_N} \geq  \frac{\log 10}{\log 12} +\log 2\left(\frac{1}{\log 2}-\frac{2}{\log 12}\right).$$
\end{lma}

\begin{proof}
For all $\w \in \Omega^{k(\delta_N)}$ and $\u \in  \{(1,1), (2,1)\}^{l(\delta_N)-2k(\delta_N)}$, $\w(1,2)^{k(\delta_N)}\u \in \Sigma_{l(\delta_N)}$. In particular for any $\w \in \Omega^{k(\delta_N)}$,
\begin{equation} \label{M-lb}
M(\w, l(\delta_N)) =2^{l(\delta_N)-2k(\delta_N)} \approx 2^{(\frac{\log 12}{\log 2}-2)13^N},
\end{equation}
noting that $\frac{\log 12}{\log 2}>2$. Hence
\begin{eqnarray*}
N_{\delta_N}(\Pi(\Sigma)) &\geq& N_{\delta_N}\left(\bigcup_{\w \in \Omega^{k(\delta_N)}}\Pi([\w])\right)\\
&\approx& \sum_{\w \in \Omega^{k(\delta_N)}} N_{\delta_N}\left(\Pi([\w])\right)\\
&\approx& \sum_{\w \in \Omega^{k(\delta_N)}}M(\w, l(\delta_N)) \\
&\approx&10^{13^N}2^{(\frac{\log 12}{\log 2}-2)13^N}.
\end{eqnarray*}
Hence for some uniform constant $c>0$,
\begin{eqnarray*}
\frac{\log N_{\delta_N}(\Pi(\Sigma))}{-\log \delta_N} &\geq& \frac{13^N \log 10}{13^N \log 12} +\frac{13^N (\frac{\log 12}{\log 2}-2)\log 2}{13^N \log 12} +\frac{\log c}{-13^N \log 12} \\
&=& \frac{\log 10}{\log 12} +\log 2\left(\frac{1}{\log 2}-\frac{2}{\log 12}\right)+ \frac{\log c}{-13^N \log 12} .
\end{eqnarray*}
The result follows by letting $N \to \infty$.
\end{proof}

\begin{lma}[Scales with small dimension] \label{low}
$$\liminf_{N \to \infty} \frac{\log N_{\delta_N'}(\Pi(\Sigma))}{-\log \delta_N'} \leq \frac{\frac{1}{\sqrt{13}} \log 10+(1-\frac{1}{\sqrt{13}})\log 4}{\log 12}+\log 2\left(\frac{1}{\log 2}-\frac{1+\frac{1}{\sqrt{13}}}{\log 12}\right).$$
\end{lma}


\begin{proof}
Let $\varepsilon>0$. Recall that for all $N \in \N$, $-\log \delta_N'=13^{N-\frac{1}{2}}\log 12$, $k(\delta_N')=\lceil13^{N-\frac{1}{2}}\rceil$ and $l(\delta_N')=\lceil13^{N-\frac{1}{2}} \frac{\log 12}{\log 2}\rceil$. Recall that $\Sigma=\overline{B}$ where $B$ is the set of all infinite sequences which label a one-sided infinite path on the graph $G$ given in Figure 1, and where the set of points $\overline{B} \setminus B$ are characterised in \eqref{boundary}. Therefore any word $\ii \in \Sigma_{k(\delta_N')}$ has one of the following forms:
\begin{enumerate}[(a)]
\item $\ii=\u$ for $\u \in \I_{k(\delta_N')}$,
\item $\ii=\u\w$ for $\u \in \I_u$, $\w \in \Omega^w$ where $u+w=k(\delta_N')$,
\item $\ii=\w$ for $\w \in \Omega^{k(\delta_N')}$,
\item $\ii=\w(1,2)^z$ for $\w \in \Omega^w$ where $w+z=k(\delta_N')$,
\item $\ii=\u\w(1,2)^z$ for $\u \in \I_u$ and $\w \in \Omega^w$ where $u+w+z=k(\delta_N')$ and $z \leq w$.
\end{enumerate}
Let $Y_a\subset \Sigma_{k(\delta_N')}$ be the set of words which have the form (a) and let $X_a \subset \Sigma$ be the subset $\{\i \in \Sigma: \i|k(\delta_N') \in Y_a\}$. Define $X_b,X_c,X_d,X_e$ and $Y_b,Y_c,Y_d,Y_e$ analogously. We note that these sets are not all mutually exclusive, for example $Y_a \cap Y_e \neq \emptyset$, but this will not affect our bounds.

\noindent \emph{Upper bound on $N_{\delta_N'}(\Pi(X_a))$.}  For any $\jj \in \{(1,1), (2,1)\}^{l(\delta_N')-k(\delta_N')}$ and $\u \in \I_{k(\delta_N')}$, $\u \jj \in \Sigma^{l(\delta_N')}$.
Therefore for each $\u \in  \I_{k(\delta_N')}$, 
\begin{equation}
M(\u, l(\delta_N')) =2^{l(\delta_N')-k(\delta_N')} \approx 2^{13^{N-\frac{1}{2}}(\frac{\log 12}{\log 2}-1)}.\label{Ma}
\end{equation}
Hence
\begin{eqnarray*}
N_{\delta_N'}(\Pi(X_a)) &\approx& \sum_{\u \in Y_a} N_{\delta_N'}(\Pi([\u])) \\
&\approx& \sum_{\u \in \I_{k(\delta_N')}} M(\u, l(\delta_N')) \\
&\lesssim_\varepsilon& (4e^\varepsilon)^{13^{N-\frac{1}{2}}}2^{13^{N-\frac{1}{2}}(\frac{\log 12}{\log 2}-1)}
\end{eqnarray*}
  by Lemma \ref{I} and \eqref{Ma}. Since $\varepsilon>0$ was chosen arbitrarily and $-\log \delta_N'=13^{N-\frac{1}{2}}\log 12$, we deduce that
\begin{equation} \label{a}
\liminf_{N \to \infty} \frac{\log N_{\delta_N'}(\Pi(X_a))}{-\log \delta_N'} \leq \frac{\log 4}{\log 12}+\log 2\left(\frac{1}{\log 2}-\frac{1}{\log 12}\right).
\end{equation}

\noindent \emph{Upper bound on $N_{\delta_N'}(\Pi(X_c))$.} Suppose $\i \in X_c$ so $\i|k(\delta_N')=\w \in \Omega^{k(\delta_N')}$. By definition of $\Sigma$, either $\i \in \Omega^\N$ or $\i$ begins with $\u(1,2)^z$ for some $\u \in \Omega^*$  where $|\u| \geq k(\delta_N')=\lceil13^{N-\frac{1}{2}}\rceil$ and $z \geq 13^N$. For $N$ sufficiently large
$$z+|\u|\geq 13^N+13^{N-\frac{1}{2}}>13^{\frac{1}{2}}13^{N-\frac{1}{2}}>\lceil\frac{\log 12}{\log 2} 13^{N-\frac{1}{2}}\rceil=l(\delta_N').$$ 
In particular for any $\w \in\Omega^{k(\delta_N')}$,
\begin{equation} \label{Mc}
M(\w, l(\delta_N'))=1.
\end{equation}
By \eqref{Mc} 
\begin{eqnarray*}
N_{\delta_N'}(\Pi(X_c)) &\approx& \sum_{\w \in Y_c} N_{\delta_N'}(\Pi([\w])) \\
&\approx& \sum_{\w \in \Omega^{k(\delta_N')}} M(\w,l(\delta_N'))=10^{k(\delta_N')}\approx10^{13^{N-\frac{1}{2}}}.
\end{eqnarray*}
Therefore since $-\log \delta_N'=13^{N-\frac{1}{2}}\log 12$,
\begin{equation} \label{c}\liminf_{N \to \infty}\frac{\log N_{\delta_N'}(\Pi(X_c))}{-\log \delta_N'} \leq \frac{\log 10}{\log 12}.
\end{equation}

\noindent \emph{Upper bound on $N_{\delta_N'}(\Pi(X_d))$.}  For $x>0$ we let $\Tau(x)$ denote the smallest power of $13$ which is greater than or equal to $x$. Suppose $\i \in X_d$, so that $\i|k(\delta_N')=\w(1,2)^z$ for $\w \in \Omega^w$ where $w+z=k(\delta_N')$. Either $\i=\w(1,2)^\infty$ or $\i$ begins with $\w(1,2)^{z'}\jj$ for some $\jj \in \Sigma_1 \setminus\{ (1,2)\}$ and 
$$z' \geq \Tau(\max\{w,z\})=\Tau(\max\{w,k(\delta_N')-w\})=13^N$$
where the final equality is because for sufficiently large $N$
$$\max\{w,k(\delta_N')-w\} \geq \frac{k(\delta_N')}{2}=\frac{\lceil13^{N-\frac{1}{2}}\rceil}{2}>13^{N-1}.$$
 Moreover, for sufficiently large $N$
$$w+z' \geq 13^N>\lceil13^{N-\frac{1}{2}} \frac{\log 12}{\log 2}\rceil=l(\delta_N').$$ 
In particular, for any $\w(1,2)^z \in Y_d$,
\begin{equation}
M\left(\w(1,2)^{z}\;,\;l(\delta_N')\right)=1.
\label{Md}
\end{equation}
By \eqref{Md} 
\begin{eqnarray*}
N_{\delta_N'}(\Pi(X_d)) &\approx&\sum_{\ii \in Y_d} N_{\delta_N'}(\Pi([\ii])) \\
&\approx&\sum_{w=1}^{k(\delta_N')-1}\sum_{\w \in \Omega^w} M\left(\w(1,2)^{k(\delta_N')-w}\;,\;l(\delta_N')\right)\\
&\lesssim_\varepsilon& (10e^\varepsilon)^{k(\delta_N')}\approx(10e^\varepsilon)^{13^{N-\frac{1}{2}}}.
\end{eqnarray*}
 Since $\varepsilon>0$ was arbitrary and  $-\log \delta_N'=13^{N-\frac{1}{2}}\log 12$,
\begin{equation} \label{d}\liminf_{N \to \infty}\frac{\log N_{\delta_N'}(\Pi(X_d))}{-\log \delta_N'} \leq \frac{\log 10}{\log 12}.
\end{equation}

\noindent \emph{Upper bound on $N_{\delta_N'}(\Pi(X_b))$.} Suppose $\i \in X_b$, so $\i|k(\delta_N') =\u\w$ for $\u \in \I_u, \w \in \Omega^w$ where $u+w=k(\delta_N')$. Either $\i=\u\j$ where $\j \in \Omega^\N$ or $\i$ begins with $\u\v(1,2)^z$ where $\v \in \Omega^z$ with $z =|\v| = \Tau(|\v|) \geq \Tau(w)$. In particular, for any $\u\w\in Y_b$ 
\begin{equation} \label{Mb}
M(\u\w,l(\delta_N')) \leq 2^{l(\delta_N')-|\u|-|\v|-z}= 2^{l(\delta_N')-|\u|-2z}=2^{l(\delta_N')-k(\delta_N')+w-2z} \leq 2^{l(\delta_N')-k(\delta_N')+w-2\cdot\Tau(w)}.
\end{equation}
 By \eqref{Mb} and Lemma \ref{I}
\begin{align}
N_{\delta_N'}(\Pi(X_b)) &\approx \sum_{\ii \in Y_b} N_{\delta_N'}(\Pi([\ii]))  \nonumber\\
&\approx \sum_{w=1}^{k(\delta_N')-1} \sum_{\u \in \I^{k(\delta_N')-w}}\sum_{\w \in \Omega^{w}} M(\u\w, l(\delta_N'))\nonumber\\
&\lesssim_\varepsilon \sum_{w=1}^{k(\delta_N')-1} 10^{w}(4e^\varepsilon)^{k(\delta_N')-w} 2^{l(\delta_N')-k(\delta_N')+w-2\Tau(w)} \nonumber\\
&\leq \sum_{w=1}^{13^{N-1}} \left(\frac{10 \cdot 2}{4e^\varepsilon \cdot 2^2}\right)^{w} (4e^\varepsilon)^{k(\delta_N')} 2^{l(\delta_N')-k(\delta_N')}+ \label{2sums1} \\
&  \sum_{w=13^{N-1}+1}^{k(\delta_N')-1}\left(\frac{10 \cdot 2}{4e^\varepsilon \cdot 2^{2\sqrt{13}}}\right)^{w} (4e^\varepsilon)^{k(\delta_N')} 2^{l(\delta_N')-k(\delta_N')} \label{2sums}
\end{align}
where in the first sum \eqref{2sums1} we have used the trivial lower bound $\Tau(x) \geq x$ and in the second sum \eqref{2sums} we have used that for all $13^{N-1}+1 \leq x\leq k(\delta_N')-1=\lceil 13^{N-1/2}\rceil-1$,
$$\sqrt{13} x \leq \sqrt{13}(13^{N-\frac{1}{2}}-1) \leq 13^N=\Tau(x).$$
For sufficiently small $\varepsilon>0$, the first sum \eqref{2sums1} can be bounded above by
$$ \sum_{w=1}^{13^{N-1}} \left(\frac{10 \cdot 2}{4e^\varepsilon \cdot 2^2}\right)^{w} (4e^\varepsilon)^{k(\delta_N')} 2^{l(\delta_N')-k(\delta_N')} \lesssim_\varepsilon 10^{13^{N-1}} (4e^{2\varepsilon})^{k(\delta_N')-13^{N-1}}2^{l(\delta_N')-k(\delta_N')-13^{N-1}}.$$
For sufficiently small $\varepsilon>0$,
$$\frac{10 \cdot 2}{4 e^\varepsilon\cdot 2^{2\sqrt{13}}}=\frac{5}{e^\varepsilon4^{\sqrt{13}}}<1$$
hence the second sum \eqref{2sums} can be bounded above by 
\begin{eqnarray*}
\sum_{w=13^{N-1}+1}^{k(\delta_N')-1}\left(\frac{10 \cdot 2}{4 e^\varepsilon\cdot 2^{2\sqrt{13}}}\right)^{w} (4e^\varepsilon)^{k(\delta_N')} 2^{l(\delta_N')-k(\delta_N')} &\lesssim_\varepsilon&
10^{13^{N-1}} (4e^{2\varepsilon})^{k(\delta_N')-13^{N-1}}2^{l(\delta_N')-k(\delta_N')+13^{N-1}-2\cdot13^{N-\frac{1}{2}}}\\
&<&10^{13^{N-1}} (4e^{2\varepsilon})^{k(\delta_N')-13^{N-1}}2^{l(\delta_N')-k(\delta_N')-13^{N-1}}.\end{eqnarray*}
In particular
\begin{align*}
N_{\delta_N'}(\Pi(X_b))& \lesssim_\varepsilon 10^{13^{N-1}} (4e^{2\varepsilon})^{k(\delta_N')-13^{N-1}}2^{l(\delta_N')-k(\delta_N')-13^{N-1}}\\
&\approx10^{13^{N-1}} (4e^{2\varepsilon})^{13^{N-\frac{1}{2}}-13^{N-1}}2^{(\frac{\log 12}{\log 2}-1)13^{N-\frac{1}{2}}-13^{N-1}}.\end{align*}
Since $\varepsilon>0$ was chosen arbitrarily and $-\log \delta_N'=13^{N-\frac{1}{2}}\log 12$,
\begin{equation} \label{b}\liminf_{N \to \infty}\frac{\log N_{\delta_N'}(\Pi(X_b))}{-\log \delta_N'} \leq \frac{\frac{1}{\sqrt{13}} \log 10+(1-\frac{1}{\sqrt{13}})\log 4}{\log 12}+\log 2\left(\frac{1}{\log 2}-\frac{1+\frac{1}{\sqrt{13}}}{\log 12}\right).
\end{equation}

\noindent \emph{Upper bound on $N_{\delta_N'}(\Pi(X_e))$.} If $\u\w(1,2)^z \in Y_e$ with $|\w|=w$ and $|\u|=u$ then since $u+w \leq k(\delta_N')=\lceil 13^{N-\frac{1}{2}}\rceil$ we have 
$$l(\delta_N')-2w-u \geq l(\delta_N')- 2\lceil 13^{N-\frac{1}{2}}\rceil> l(\delta_N')-\lceil \frac{\log 12}{\log 2} 13^{N-\frac{1}{2}}\rceil=0.$$
In particular
\begin{equation}\label{Me}
M(\u\w(1,2)^z, l(\delta_N'))=2^{l(\delta_N')-2w-u}.
\end{equation} By \eqref{Me} and Lemma \ref{I},
\begin{eqnarray*}
N_{\delta_N'}(\Pi(X_e)) &\approx& \sum_{\ii \in Y_e} N_{\delta_N'}(\Pi([\ii])) \\
&\approx & \sum_{w=13^r\leq 13^{N-1}}\sum_{u=1}^{k(\delta_N')-w-1} \sum_{\u \in \I^{u}}\sum_{\w \in \Omega^w} M(\u\w(1,2)^{k(\delta_N')-u-w},l(\delta_N')) \\
&\lesssim_\varepsilon &\sum_{w=13^r\leq 13^{N-1}}\sum_{u=1}^{k(\delta_N')-w-1} (4e^\varepsilon)^{u}10^{w}2^{l(\delta_N')-2w-u}\\
& \lesssim_\varepsilon& \sum_{w=13^r\leq 13^{N-1}}\left(\frac{10 \cdot 2}{4e^{2\varepsilon} \cdot 2^2}\right)^{w} \left(\frac{4e^{2\varepsilon}}{2}\right)^{k(\delta_N')}2^{l(\delta_N')} \\
&\lesssim_\varepsilon& 10^{13^{N-1}} (4e^{3\varepsilon})^{k(\delta_N')-13^{N-1}}2^{l(\delta_N')-k(\delta_N')-13^{N-1}}\\
&\approx&10^{13^{N-1}} (4e^{3\varepsilon})^{13^{N-\frac{1}{2}}-13^{N-1}}2^{(\frac{\log 12}{\log 2}-1)13^{N-\frac{1}{2}}-13^{N-1}} .
\end{eqnarray*}
Since $\varepsilon>0$ was chosen arbitrarily and  $-\log \delta_N'=13^{N-\frac{1}{2}}\log 12$,
\begin{equation} \label{e}\liminf_{N \to \infty}\frac{\log N_{\delta_N'}(\Pi(X_e))}{-\log \delta_N'} \leq \frac{\frac{1}{\sqrt{13}} \log 10+(1-\frac{1}{\sqrt{13}})\log 4}{\log 12}+\log 2\left(\frac{1}{\log 2}-\frac{1+\frac{1}{\sqrt{13}}}{\log 12}\right).
\end{equation}

Since the upper bounds in \eqref{b} and \eqref{e} are strictly greater than the upper bounds in \eqref{a}, \eqref{c} and \eqref{d} the proof is complete.
\end{proof}

\begin{proof}[Proof of Theorem \ref{main}]
 $\Pi(\Sigma)$ is invariant under the smooth expanding map $T(x,y)=(mx \mod 1, ny \mod 1)$. Note that to four decimal places
$$ \frac{\log 10}{\log 12} +\log 2\left(\frac{1}{\log 2}-\frac{2}{\log 12}\right) \approx 1.3687 $$
and
$$ \frac{\frac{1}{\sqrt{13}} \log 10+(1-\frac{1}{\sqrt{13}})\log 4}{\log 12}+\log 2\left(\frac{1}{\log 2}-\frac{1+\frac{1}{\sqrt{13}}}{\log 12}\right)\approx 1.3038.$$
 By Lemmas \ref{high} and \ref{low}
$$
\ubd \Pi(\Sigma)\geq \limsup_{N \to \infty} \frac{\log N_{\delta_N}(\Pi(\Sigma))}{-\log \delta_N}>\liminf_{N \to \infty} \frac{\log N_{\delta_N'}(\Pi(\Sigma))}{-\log \delta_N'}\geq\lbd \Pi(\Sigma).$$
In particular, the box dimension of $\Pi(\Sigma)$ does not exist.
\end{proof}

\begin{rem}
Lemmas \ref{high} and \ref{low} can also be used to demonstrate the existence of infinitely generated self-affine sets whose box dimensions are distinct.  Consider the countable family of affine contractions 
$$\{S_{(1,1)}\} \cup \{S_{(1,2)}\} \cup \bigcup_{N=1}^\infty \bigcup_{\w \in \Omega^{13^N}} \{S_{\w(1,2)^{13^N}}\}$$
which generates the infinitely generated self-affine set $E=\Pi(\tilde{\Sigma})$ where
$$\tilde{\Sigma}:= \{\u_1\u_2 \ldots \, :  \,\u_i \in \mathcal{C}\; \textnormal{for all $i \in \N$}\}.$$
Since $E \subset F$, $\lbd E \leq \lbd F$. On the other hand, for all $N \in \N$, $\w \in \Omega^{k(\delta_N)}$ and $\u \in  \{(1,1), (2,1)\}^{l(\delta_N)-2k(\delta_N)}$, 
$$[\w(1,2)^{13^N}\u] \cap \tilde{\Sigma} \neq \emptyset.$$ Therefore by bounding $N_{\delta_N}(E)$ in the same way as in Lemma \ref{high} we deduce that $\lbd E< \ubd E$.
\end{rem}

\section{Further questions} \label{further}

Here we suggest some possible directions for future work.

\begin{ques} Does there exist an expanding repeller whose box dimension does not exist? Namely, does there exist a smooth expanding map $f:M \to M$ of a Riemannian manifold $M$ and compact set $\Lambda=f(\Lambda)$ such that $\Lambda=\{x \in U: f^n(x) \in U,\; \forall n \in \N\}$ for some open neighbourhood $U$ of $\Lambda$?
\end{ques}

\begin{ques} Given a smooth diffeomorphism $f:M \to M$, does the box dimension of its basic set (or intersections of the basic set with local stable and unstable manifolds) always exist? 
\end{ques}

\end{document}